\documentclass{amsart}
\usepackage{amsfonts,amssymb,amsmath,amsthm}
\usepackage{url}
\usepackage{enumerate}
 \usepackage{algorithm}
\usepackage[noend]{algpseudocode}
\usepackage{float}
\usepackage{tikz, framed}
\usetikzlibrary{arrows,shapes,matrix,decorations.pathmorphing}
\usetikzlibrary{backgrounds}
\newtheorem{thrm}{Theorem}[section]

\newtheorem{cor}[thrm]{Corollary}
\theoremstyle{definition}
\newtheorem{definition}[thrm]{Definition}
\newtheorem{remark}[thrm]{Remark}
\newtheorem{example}[thrm]{Example}
\newtheorem{Proposition}[thrm]{Proposition}
\numberwithin{equation}{section}

 \newcommand{\NN}{\mathbb{N}}

  \newcommand{\cT}{{\sf{T}}}
 \newcommand{\cN}{{\sf{N}}}
  \newcommand{\cB}{{\sf{B}}}

\newcommand{\bM}{{\overline{M}}}

  \newcommand{\cG}{{\sf{G}}}

  \newcommand{\ck}{{\mathbf{k}}}
  \newcommand{\st}{{\textrm{such that\ }}}

 \newcommand{\kT}{\mathcal{T}}
\newcommand{\kN}{\mathcal{N}}

\newcommand{\kA}{\mathcal{A}}
\newcommand{\kM}{\mathcal{M}}
\newcommand{\kF}{\mathcal{F}}
\newcommand{\kP}{\mathcal{P}}

\author{Michela Ceria}
\address{
Department of Computer Science\\
University of Milan\\
Milan, Italy.}
\email{ michela.ceria@gmail.com}
\thanks{ }


\begin{document}

\title[Applications of Bar Code to involutive divisions]{Applications of Bar Code to involutive divisions  and a greedy algorithm for complete sets.}

\begin{abstract}
In this paper, we describe 
  how to get Janet decomposition for a finite set of terms and detect completeness
  of that set by means of the associated Bar Code. Moreover,   we explain an algorithm to find a variable ordering (if it exists) s.t. a given set of terms is complete according to that ordering. The algorithm is greedy and constructs a Bar Code from the maximal to the minimal variable, adjusting the variable ordering with a sort of backtracking technique, thus allowing to construct the desired ordering without trying all the $n!$ possible orderings.\end{abstract}
\maketitle

\section{Introduction}\label{Intro}
Let  $\mathcal{P}:=\mathbf{k}[x_1,...,x_n]$ be the polynomial ring in 
$n$ variables with coefficients in the field $\ck$. The \emph{semigroup of terms}, generated by  $\{x_1,...,x_n\}$ is:
$\mathcal{T}:=\{x^{\gamma}:=x_1^{\gamma_1}\cdots
x_n^{\gamma_n} \vert \,\gamma:=(\gamma_1,...,\gamma_n)\in \NN^n \}.$
Given a monomial/semigroup ideal $J\subset\mathcal{T}$ and its minimal set of generators  ${\sf G}(J)$ 
  Janet introduced in  \cite{J1} the notion of \emph{multiplicative variables} and the connected decomposition of $J$ into disjoint 
  \emph{cones}, giving a procedure (\emph{completion})   
  to construct such a decomposition.
 In particular, $\forall v\in\mathcal{T}$, there is a \emph{unique} decomposition 
  $v=tu$, with $t \in {\sf  G}(J)$ and
$u$ a  product of powers of $t$'s 
  multiplicative variables.  
While performing reduction w.r.t. an ideal whose initial ideal is $J$, the term $w$ can be reduced by the only polynomial whose leading term generates the cone containing $w$. 
Involutive divisions date back to the works by Janet \cite{J1,J2,J3,J4} who, besides giving a cone decomposition for the monomial ideal $J$, in order to describe Riquier's \cite{Riq} formulation 
of the description for the general solutions of a PDE problem, gave a similar decomposition also for the related \emph{escalier}
${\sf{N}}(J):=\mathcal{T}\setminus J.$ 
 Later in \cite{J2,J3,J4}, he gave a new decomposition (and an algorithm to produce it) which called \emph{involutive} and which is
 behind  both Gerdt-Blinkov \cite{GB1, GB2, GB3} procedure to compute Gr\"obner bases and Seiler's \cite{SeiB} involutiveness theory. His aim was twofold: to reinterpret, in terms of multiplicative 
variables and cone decomposition, Cartan's solution to PDE problems \cite{Car1, Car2, Car3} (whence the name \emph{involutiveness}) and  to re-evaluate within his theory the notion of \emph{generic 
initial ideal} introduced by Delassus \cite{Del1, Del2, Del3} and the correction of his mistake by Robinson \cite{Rob1, Rob2} and Gunther \cite{Gun1, Gun2}, who remark that the notion requires $J$ to 
be  Borel-fixed 
 (an equivalent modern reformulation was proposed by Galligo  \cite{GAL}, who merged Hironaka and Grauert's ideas  \cite{Hir,Gra}; see also \cite{GS,Ei}). Janet remarked that all Borel-fixed ideals 
are involutive, but the converse is false. More precisely, in \cite{J3} Janet presents, as \emph{nouvelle formes cannoniques}, Delassus, Robinson and Gunther's results and compares them with the one 
deductible from an involutive basis and in  \cite[p.62]{J4}, given a homogeneous  ideal $\mathcal{I} \triangleleft \mathcal{P}$ in generic coordinates,  he restates Riquier's completion in terms of a 
Macaulay-like construction, iteratively computing the vector spaces $\mathcal{I}_d \!:=\!\{f\in\mathcal{I} \!:\! \deg(f) = d\}$ until Cartan test grants that Castelnuovo-Mumford \cite[pg.99]{Mum} 
regularity $D$ has been reached. This would allow him to consider the  semigroup ideal ${\sf T}(\mathcal{I})$ of the leading terms w.r.t.  deg-lex 
(in the sense of Gr\"obner basis theory) and get the \emph{involutive reduction} required by Riquier's procedure.
 The formal definition of involutive division is due to Gerdt-Blinkov \cite{GB1, GB2}.
\\
Bar Codes, introduced in \cite{Ce,BCStr}, are a visual representation
for finite sets of terms $M\subset \kT$. In particular, if $M=\cN(I)$
is the Groebner escalier of a zerodimensional ideal $I\triangleleft \kP$, many of its properties can be directly deduced by its Bar Code.
As an example, in \cite{CMCeMu}, Bar Codes are employed to develop a combinatorial algorithm which, given a finite set
of distinct points, computes the lexicographical Groebner escalier of its vanishing ideal.
This algorithm is an alternative to those by Cerlienco-Mureddu \cite{CeMu, CeMu2, CeMu3} and by Felszeghy-R\'{a}th-R\'{o}nyay \cite{FRR}, which keeps the former algorithm's iterativity, though reaching a complexity which is near to that of the latter one. 
In \cite{Ce}, we use Bar Codes as tools to define a bijection between zerodimensional (strongly) stable ideals in two or three variables and some partitions of their (constant) affine Hilbert polynomial.
\\
Now, we are focusing on the properties of Bar Codes connected to involutive divisions.
Bar Codes are a good technology to study involutive divisions.
For example, it is trivial to compute the Pommaret \cite{J2} basis of $I$ from the Bar Code.
In \cite{NotaBasi}, we exploit the Bar Code to compute by Moeller interpolation the Pommaret basis of the ideal of a finite set of distinct points.  For a general overview of Bar Codes' applications see \cite{BCStr}.
\\
In this paper, we discuss some applications of the Bar Code to involutive divisions. In particular, we see how 
the Bar Code associated to a finite set of terms, which is non-necessarily an order ideal, allows to approach Janet decomposition \cite{J1} and decide whether that set is complete according to Janet's definition.
Moreover, we give an algorithm to check whether there is a variables' ordering s.t. a given set $M\subset \kT$
is complete.  We need to remark that such a topic has some connections to the study of Stanley decompositions and Stanley depth. 
Indeed, Janet decomposition for a complete set is exactly a Stanley decomposition which can be easily read off from that set. 
Anyway, has stated by Herzog \cite{HS},
\begin{quote}
{\em Janet decompositions from the viewpoint of Stanley depth
are not optimal. They rarely give Stanley decompositions providing the Stanley depth of a monomial ideal. However one obtains the result that the Stanley depth of a monomial ideal is at least 1.}
\end{quote}
and, actually, this paper places itself in the field of study mainly developed by Gerdt-Blinkov \cite{GB1,GB2,GB3} and Seiler \cite{SeiB}, which has aims and language that are different from those of Stanley depth.

\noindent After the next section, devoted to notation, we describe the Bar Code (section \ref{BCsect}),
as the fundamental tool for the following sections.
Then, in section \ref{Jdecsect}, we describe Janet decomposition into multiplicative/non-multiplicative variables
 and we explain how to use the Bar Code to get it from a finite set of terms. Moreover, we also deal with complete sets,
 explaining how also completeness can be read from a suitable Bar Code.
 In section \ref{greedysect}, then, we explain an algorithm to detect a variable ordering (if it exists) 
 s.t. a given set of terms is complete according to that ordering. The algorithm is greedy and constructs
 a Bar Code from the maximal to the minimal variable, adjusting the variable ordering with a sort of backtracking technique, and 
 allowing to construct the desired ordering without trying all the $n!$ possible orderings.

\section{Some general notation}\label{Notat}
Throughout this paper we mainly follow the notation of \cite{SPES}. We denote by $\mathcal{P}:=\mathbf{k}[x_1,...,x_n]$ the ring of polynomials in
$n$ variables with coefficients in the field $\ck$. The \emph{semigroup of terms}, generated by the set $\{x_1,...,x_n\}$ is:
$$\mathcal{T}:=\{x^{\gamma}:=x_1^{\gamma_1}\cdots
x_n^{\gamma_n} \vert \,\gamma:=(\gamma_1,...,\gamma_n)\in \NN^n \}.$$ If $\alpha\subseteq \{1,...,n\}$ then $\kT[\alpha]:=\{x^{\gamma}:=x_1^{\gamma_1}\cdots
x_n^{\gamma_n} \in \kT \vert \,\gamma_i \neq 0 \Rightarrow i \in \alpha\}.$ If $t=x_1^{\gamma_1}\cdots x_n^{\gamma_n}$, then $\deg(t)=\sum_{i=1}^n
\gamma_i$ is the \emph{degree} of $t$ and, for each $h\in \{1,...,n\}$
$\deg_h(t):=\gamma_h$ is the $h$-\emph{degree} of $t$.
 A \emph{semigroup ordering} $<$ on $\mathcal{T}$  is  a total ordering
\st $ t_1<t_2 \Rightarrow st_1<st_2,\, \forall s,t_1,t_2
\in \mathcal{T}.$ For each semigroup ordering $<$ on $\mathcal{T}$,  we can represent a polynomial
$f\in \mathcal{P}$ as a linear combination of terms arranged w.r.t. $<$, with
coefficients in the base field $\mathbf{k}$:
$$f=\sum_{t \in \mathcal{T}}c(f,t)t=\sum_{i=1}^s c(f,t_i)t_i:\,
c(f,t_i)\in
\mathbf{k}\setminus \{0\},\, t_i\in \mathcal{T},\, t_1>...>t_s,$$ with
$\cT(f):=t_1$   the 
\emph{leading term} of $f$, $Lc(f):=c(f,t_1)$ the  \emph{leading
coefficient} 
of $f$ and $tail(f):=f-c(f,\cT(f))\cT(f)$  the 
\emph{tail} of $f$.
A \emph{term ordering} is a semigroup ordering \st $1$ is lower 
than every variable or, equivalently, it is a \emph{well ordering}.
In all paper, we consider the \emph{lexicographical ordering} 
induced
by  $x_1<...<x_n$, i.e:
$ x_1^{\gamma_1}\cdots x_n^{\gamma_n}<_{Lex} x_1^{\delta_1}\cdots
x_n^{\delta_n} \Leftrightarrow \exists j\, \vert  \,
\gamma_j<\delta_j,\,\gamma_i=\delta_i,\, \forall i>j, $
which is a term ordering. Since we do not consider any 
term ordering other than Lex, we drop the subscript and denote it by $<$ 
instead of $<_{Lex}$.\\
A subset $J \subseteq \kT$ is a \emph{semigroup ideal} if  $t \in J 
\Rightarrow st \in J,\, \forall s \in \mathcal{T}$; a subset ${\sf N}\subseteq \mathcal{T}$ is an \emph{order ideal} if
$t\in {\sf N} \Rightarrow s \in {\sf N}\, \forall s \vert t$.  
We have that ${\sf N}\subseteq \mathcal{T}$ is an order ideal if and only if 
$\mathcal{T}\setminus {\sf N}=J$ is a semigroup ideal.
Given a semigroup ideal $J\subset\mathcal{T}$  we define ${\sf 
N}(J):=\mathcal{T}\setminus J$. The minimal set of generators ${\sf G}(J)$ of $J$ is called the \emph{monomial basis} 
of $J$.
 For all subsets $G \subset \mathcal{P}$,  $\cT\{G\}:=\{\cT(g),\, g \in  G\}$ and $\cT(G)$ is the semigroup ideal
of leading terms defined as $\cT(G):=\{t \cT(g),\, t \in \mathcal{T}, g \in G\}$. Fixed a term order $<$, for any ideal  $I
\triangleleft \mathcal{P}$ the monomial basis of the semigroup ideal 
$\cT(I)=\cT\{I\}$ is called \emph{monomial basis}  of $I$ and denoted again by $\cG(I)$,
whereas the ideal 
$In(I):=(\cT(I))$ is called \emph{initial ideal} and the order ideal 
$\cN(I):=\kT \setminus \cT(I)$ is called \emph{Groebner escalier} of $I$.

\section{Bar Code for monomial ideals}\label{BCsect}
In this section, referring to \cite{Ce, BCStr}, we summarize the main definitions and properties about Bar Codes, which will be used in what follows. First of all, we recall the general definition of  Bar Code. 
\begin{definition}\label{BCdef1}
A Bar Code $\cB$ is a picture composed by segments, called \emph{bars}, 
superimposed in horizontal rows, which satisfies conditions $a.,b.$ below.
Denote by 
\begin{itemize}
 \item $\cB_j^{(i)}$ the $j$-th bar (from left to right) of the $i$-th row 
 (from top to bottom), $1\leq i \leq n$, i.e. the \emph{$j$-th $i$-bar};
 \item $\mu(i)$ the number of bars of the $i$-th row
 \item $l_1(\cB_j^{(1)}):=1$, $\forall j \in \{1,2,...,\mu(1)\}$ the $(1-)$\emph{length} of the $1$-bars;
 \item $l_i(\cB_j^{(k)})$, $2\leq k \leq n$, $1 \leq i \leq k-1$, $1\leq j \leq \mu(k)$ the $i$-\emph{length} of $\cB_j^{(k)}$, i.e. the number of $i$-bars lying over $\cB_j^{(k)}$
\end{itemize}
\begin{itemize}
 \item[a.] $\forall i,j$, $1 \leq i \leq n-1$, $1\leq j \leq \mu(i)$, $\exists ! \overline{j}\in \{1,...,\mu(i+1)\}$ s.t. $\cB_{\overline{j}}^{(i+1)}$ lies  under  $\cB_j^{(i)}$ 
 \item[b.] $\forall i_1,\,i_2 \in \{1,...,n\}$, $\sum_{j_1=1}^{\mu(i_1)} l_1(\cB_{j_1}^{(i_1)})= \sum_{j_2=1}^{\mu(i_2)} l_1(\cB_{j_2}^{(i_2)})$; we will then say that  \emph{all the rows have the 
same length}.
\end{itemize}
\end{definition}
\begin{example}\label{BC1}
 An example of Bar Code $\cB$ is
\\
\begin{minipage}{5cm}
 \begin{center}
\begin{tikzpicture}
\node at (3.8,-0.5) [] {${\scriptscriptstyle 1}$};
\node at (3.8,-1) [] {${\scriptscriptstyle 2}$};
\node at (3.8,-1.5) [] {${\scriptscriptstyle 3}$};

\draw [thick] (4,-0.5) --(4.5,-0.5);
\draw [thick] (5,-0.5) --(5.5,-0.5);
\draw [thick] (6,-0.5) --(6.5,-0.5);
\draw [thick] (7,-0.5) --(7.5,-0.5);
\draw [thick] (8,-0.5) --(8.5,-0.5);
\draw [thick] (4,-1)--(5.5,-1);
\draw [thick] (6,-1) --(6.5,-1);
\draw [thick] (7,-1) --(7.5,-1);
\draw [thick] (8,-1) --(8.5,-1);
\draw [thick] (4,-1.5)--(5.5,-1.5);
\draw [thick] (6,-1.5) --(8.5,-1.5);
\end{tikzpicture}
\end{center}
 
\end{minipage}
\hspace{0.4cm}
\begin{minipage}{7cm}
 $\quad$ \\
The $1$-bars have length $1$.   As regards the other rows, $l_1(\cB_1^{(2)})=2$,
$l_1(\cB_2^{(2)})=l_1(\cB_3^{(2)})= l_1(\cB_4^{(2)})=1$,
$l_2(\cB_1^{(3)})=1$,$l_1(\cB_1^{(3)})=2$ and
\end{minipage}
\\
 $l_2(\cB_2^{(3)})=l_1(\cB_2^{(3)})=3$, so
 $\sum_{j_1=1}^{\mu(1)} l_1(\cB_{j_1}^{(1)})= \sum_{j_2=1}^{\mu(2)}
l_1(\cB_{j_2}^{(2)})= \sum_{j_3=1}^{\mu(3)} l_1(\cB_{j_3}^{(3)})=5.$
\end{example}

\noindent We outline now the construction of the Bar Code associated to a finite set 
of terms. For more details, see \cite{BCStr}, while for an alternative construction, see \cite{Ce}.\\
First of all, given a  term  $t=x_1^{\gamma_1}\cdots 
x_n^{\gamma_n} \in \mathcal{T}\subset \ck[x_1,...,x_n]$, for each $i \in \{1,...,n\}$, we take
$\pi^i(t):=x_i^{\gamma_i}\cdots x_n^{\gamma_n} \in \mathcal{T}.$ 
Taken a finite set of terms $M\subset \mathcal{T}$, for each $ i \in \{1,...,n\}$, we then define 
$M^{[i]}:=\pi^i(M):=\{\pi^i(t) \vert t \in M\}.$ \\
Now we take $M\subseteq \mathcal{T}$, with $\vert M\vert =m < \infty$ and we order its 
elements increasingly  w.r.t. Lex, getting the list  
$\bM=[t_1,...,t_m]$. Then, we construct the sets $M^{[i]}$, and 
the corresponding lexicographically ordered lists\footnote{$\bM$ cannot contain repeated terms, while the $\bM^{[i]}$, for $1<i \leq n$, can. In case some repeated terms occur in $\bM^{[i]}$, $1<i 
\leq n$, they clearly have to be adjacent in the list, due to the 
lexicographical ordering.} $\bM^{[i]}$, for $i=1,...,n$.

\noindent We define the $n\times m $ matrix of terms $\kM$   s.t. 
its $i$-th row is $\bM^{[i]}$, $i=1,...,n$, i.e.
\[\kM:= \left(\begin{array}{cccc}
\pi^{1}(t_1)&... & \pi^{1}(t_m)\\
\pi^{2}(t_1)&... & \pi^{2}(t_m)\\
\vdots & \quad &\vdots\\
\pi^{n}(t_1)& ... & \pi^{n}(t_m)
\end{array}\right)\]
\begin{definition}\label{BarCodeDiag}
 The \emph{Bar Code diagram} $\cB$ associated to $M$ (or, equivalently, to 
$\bM$) is a 
$n\times m $ diagram, made by segments s.t. the $i$-th row of $\cB$, $1\leq 
i\leq n$  is constructed as follows:
       \begin{enumerate}
        \item take the $i$-th row of $\kM$, i.e. $\bM^{[i]}$
        \item consider all the sublists of repeated terms, i.e. $[\pi^i(t_{j_1}),\pi^i(t_{j_1 +1}),
        ...,\pi^i(t_{j_1 
+h})]$ s.t. 
        $\pi^i(t_{j_1})= \pi^i(t_{j_1 
+1})=...=\pi^i(t_{j_1 +h})$, noting that\footnote{Clearly if a term 
$\pi^i(t_{\overline{j}})$ is not 
        repeated in $\bM^{[i]}$, the sublist containing it will be only  
$[\pi_i(t_{\overline{j}})]$, i.e. $h=0$.} $0 \leq h<m$ 
        \item underline each sublist with a segment
        \item delete the terms of $\bM^{[i]}$, leaving only the segments (i.e. 
the \emph{$i$-bars}).
       \end{enumerate}
 We usually label each $1$-bar $\cB_j^{(1)}$, $j \in \{1,...,\mu(1)\}$ with the 
term $t_j \in \bM$.
\end{definition}

\noindent A Bar Code diagram is a 
Bar Code in the sense of Definition \ref{BCdef1}.

\begin{example}\label{BarCodeNoOrdId}
Given  $M=\{x_1,x_1^2,x_2x_3,x_1x_2^2x_3,x_2^3x_3\}\subset
\mathbf{k}[x_1,x_2,x_3]$, we have: the $3 \times 5 $ table on the 
left and then to the 
Bar Code on the right:
\\
\begin{minipage}[b]{0.5\linewidth}
\begin{center}
\begin{tikzpicture}

\node at (4.2,-0.5) [] {${\small x_1}$};
\node at (5.2,-0.5) [] {${\small x_1^2}$};
\node at (6.2,-0.5) [] {${\small x_2x_3}$};
\node at (7.4,-0.5) [] {${\small x_1x_2^2x_3}$};
\node at (8.6,-0.5) [] {${\small x_2^3x_3}$};

\node at (4.2,-1) [] {${\small 1}$};
\node at (5.2,-1) [] {${\small 1}$};
\node at (6.2,-1) [] {${\small x_2x_3}$};
\node at (7.4,-1) [] {${\small x_2^2x_3}$};
\node at (8.6,-1) [] {${\small x_2^3x_3}$};

\node at (4.2,-1.5) [] {${\small 1}$};
\node at (5.2,-1.5) [] {${\small1}$};
\node at (6.2,-1.5) [] {${\small x_3}$};
\node at (7.4,-1.5) [] {${\small x_3}$};
\node at (8.6,-1.5) [] {${\small x_3}$};
\end{tikzpicture}
\end{center}
\end{minipage}
\hspace{0.45cm}
\begin{minipage}[b]{0.5\linewidth}
\begin{center}
\begin{tikzpicture}
\node at (4.2,0) [] {${\small x_1}$};
\node at (5.2,0) [] {${\small x_1^2}$};
\node at (6.2,0) [] {${\small x_2x_3}$};
\node at (7.2,0) [] {${\small x_1x_2^2x_3}$};
\node at (8.2,0) [] {${\small x_2^3x_3}$};

\node at (3.8,-0.5) [] {${\scriptscriptstyle 1}$};
\node at (3.8,-1) [] {${\scriptscriptstyle 2}$};
\node at (3.8,-1.5) [] {${\scriptscriptstyle 3}$};

\draw [thick] (4,-0.5) --(4.5,-0.5);
\draw [thick] (5,-0.5) --(5.5,-0.5);
\draw [thick] (6,-0.5) --(6.5,-0.5);
\draw [thick] (7,-0.5) --(7.5,-0.5);
\draw [thick] (8,-0.5) --(8.5,-0.5);
\draw [thick] (4,-1)--(5.5,-1);
\draw [thick] (6,-1) --(6.5,-1);
\draw [thick] (7,-1) --(7.5,-1);
\draw [thick] (8,-1) --(8.5,-1);
\draw [thick] (4,-1.5)--(5.5,-1.5);
\draw [thick] (6,-1.5) --(8.5,-1.5);
\end{tikzpicture}
\end{center}
\end{minipage}
\end{example} 

Now we recall the vice versa, i.e. how to associate a finite set of terms $M_\cB$ to a given Bar 
Code $\cB$. In \cite{Ce} we first give a more general procedure to do so and then we specialize it in order 
 to have a \emph{unique} set of terms for each Bar Code. Here we give only the specialized version, so 
 we follow the steps below:
 \begin{itemize}
 \item[$\mathfrak{B}1$] consider the $n$-th row, composed by the bars 
$B^{(n)}_1,...,B^{(n)}_{\mu(n)}$. Let $l_1(B^{(n)}_j)=\ell^{(n)}_j$, 
for 
$j\in\{1,...,\mu(n)\}$. Label each bar 
$B^{(n)}_j$ with $\ell^{(n)}_j$ copies 
of $x_n^{j-1}$.
 \item[$\mathfrak{B}2$] For each $i=1,...,n-1$, $1 \leq j \leq \mu(n-i+1)$ 
 consider the bar $B^{(n-i+1)}_j$ and suppose that it has been 
 labelled by 
$\ell^{(n-i+1)}_j$ copies of a term $t$. Consider all the $(n-i)$-bars 
$B^{(n-i)}_{\overline{j}},...,B^{(n-i)}_{\overline{j}+h}$ 
  lying immediately  above  $ B^{(n-i+1)}_j$; note that $h$ satisfies 
$0\leq h\leq \mu(n-i)-\overline{j}$. 
 Denote the 1-lengths of 
$B^{(n-i)}_{\overline{j}},...,B^{(n-i)}_{\overline{j}+h}$  by  
$l_1(B^{(n-i)}_{\overline{j}})=\ell^{(n-i)}_{\overline{j}}$,...,
 $l_1(B^{(n-i)}_{\overline{j}+h})=\ell^{(n-i)}_{\overline{j}+h}$. 
 For each $0\leq k\leq h$, label  $ B^{(n-i)}_{\overline{j}+k}$ with 
$\ell^{(n-i)}_{\overline{j}+k}$ copies of $t x_{n-i}^{k}$. 
 \end{itemize}
\begin{definition}\label{Admiss}
A Bar Code $\cB$ is \emph{admissible} if the set $M$ obtained by applying 
$\mathfrak{B}1$ and $\mathfrak{B}2$ to $\cB$  is an order ideal.
\end{definition}
\noindent By definition of order 
ideal, using $\mathfrak{B}1$ and $\mathfrak{B}2$ is the only way an order 
ideal can be associated to an admissible Bar Code. 
 
\begin{definition}\label{elist}
 Given a Bar Code $\cB$, 
 let us consider a $1$-bar $B_{j_1}^{(1)}$, with $j_1 
\in \{1,...,\mu(1)\}$.
 The \emph{e-list} associated to $B_{j_1}^{(1)}$ is the $n$-tuple 
$e(B_{j_1}^{(1)}):=(b_{j_1,n},....,b_{j_1,1})$, defined as follows:
 \begin{itemize}
  \item consider the $n$-bar  $B_{j_n}^{(n)}$, lying under 
  $B_{j_1}^{(1)}$. 
The number of $n$-bars on the left of $B_{j_n}^{(n)}$ is  $b_{j_1,n}.$
  \item for each $i=1,...,n-1$, let  $B_{j_{n-i+1}}^{(n-i+1)}$ and 
$B_{j_{n-i}}^{(n-i)}$ be 
  the $(n-i+1)$-bar and the $(n-i)$-bar 
lying under $B_{j_1}^{(1)}$. Consider the $(n-i+1)$-block associated to 
$B_{j_{n-i+1}}^{(n-i+1)}$, i.e. $B_{j_{n-i+1}}^{(n-i+1)}$ and all the bars lying over it. 
The number of $(n-i)$-bars of 
the block, which lie on  the 
left of $B_{j_{n-i}}^{(n-i)}$ is $b_{j_1,n-i}.$
    \end{itemize}
\end{definition}

\begin{remark}\label{ElistExp}
 Given a Bar Code $\cB$, fix a $1$-bar  $B_{j}^{(1)}$, with $j \in 
\{1,...,\mu(1)\}$.\\
 Comparing Definition \ref{elist} and the steps $\mathfrak{B}1$ and 
 $\mathfrak{B}2$ described above, we can observe that the values of the e-list 
$e(B_j^{(1)}):=(b_{j,n},....,b_{j,1})$ are exactly the
exponents of the term 
labelling $B_{j}^{(1)}$, obtained applying $\mathfrak{B}1$ and $\mathfrak{B}2$ to
$\cB$ (compare Example \ref{BarCodeNoOrdId}).
\end{remark}

\begin{Proposition}[Admissibility criterion]\label{AdmCrit}
 A Bar Code $\cB$ is admissible if and only if, for each 
 $1$-bar $\cB_{j}^{(1)}$, $j \in \{1,...,\mu(1)\}$, the e-list 
$e(\cB_j^{(1)})=(b_{j,n},....,b_{j,1})$ satisfies the following condition: 
$\forall k \in \{1,...,n\} \textrm{ s.t. } b_{j,k}>0,\, \exists \overline{j} 
 \in \{1,...,\mu(1)\}\setminus \{j\} \textrm{ s.t. } $ $$
e(\cB_{\overline{j}}^{(1)})= (b_{j,n},...,b_{j,k+1}, (b_{j,k})-1, 
b_{j,k-1},...,b_{j,1}). $$
\qed
\end{Proposition}
\noindent Consider the  sets
$\kA_n:=\{\cB \in \mathcal{B}_n \textrm{ s.t. } \cB  \textrm{ admissible}\} $ and 
$\kN_n:=\{\cN \subset \kT,\, \vert \cN\vert < \infty \textrm{ s.t. } \cN 
\textrm{ is an order ideal}\}.$ We can define the map 
$\eta: \kA_n \rightarrow \kN_n ; \quad \cB \mapsto \cN,$
where $\cN$ is the order ideal obtained applying $\mathfrak{B}1$ and $\mathfrak{B}2$ to $\cB$,
and it can be easily proved that $\eta$ is a bijection.\\
Up to this point, we have discussed the link between Bar Codes and order ideals,
 i.e. we focused on the link between Bar Codes and Groebner escaliers of 
monomial ideals. We show now that, given a Bar Code 
$\cB$ and the order ideal  $\cN =\eta(\cB)$
it is possible to deduce a very specific generating set 
for the monomial ideal $I$ s.t. $\cN(I)=\cN$.
    \begin{definition}\label{StarSet}
 The \emph{star set} of an order ideal $\cN$ and 
 of its associated Bar Code $\cB=\eta^{-1}(\cN)$ is a set $\kF_\cN$ constructed as 
follows:
 \begin{itemize}
  \item[a)] $\forall 1 \leq i\leq n$, let $t_i$ be a term 
  which labels a $1$-bar lying over $\cB^{(i)}_{\mu(i)}$, 
  then 
  $x_i\pi^i(t_i)\in \kF_\cN$;
  \item[b)] $\forall 1 \leq i\leq n-1$, 
  $\forall 1 \leq j \leq \mu(i)-1$ let 
  $\cB^{(i)}_j$ and $\cB^{(i)}_{j+1}$ be two 
  consecutive bars not lying over the 
same $(i+1)$-bar and let $t^{(i)}_j$ be a term
which labels a $1$-bar lying 
over   $\cB^{(i)}_j$, then 
  $x_i\pi^i(t^{(i)}_j)\in \kF_\cN$.
 \end{itemize}
\end{definition}
\noindent We usually represent $\kF_\cN$ within 
the associated Bar Code $\cB$, inserting
each $t \in \kF_\cN$ on the right of the bar from which 
it is deduced.
Reading the terms from left to right and from the top to
the bottom, $\kF_\cN$ 
is ordered w.r.t. Lex. 
\begin{example}\label{BCP}
$\quad $\\
\begin{minipage}{6.5cm}
For ${\sf
N}=\{1,x_1,x_2,x_3\}\subset
\mathbf{k}[x_1,x_2,x_3]$,  
we have $\kF_\cN=\{x_1^2,x_1x_2,x_2^2,x_1x_3,x_2x_3,x_3^2\}$; looking at 
Definition \ref{StarSet}, we can see that  the terms $x_1x_3,x_2x_3,x_3^2$ come 
from a), while the terms  
$x_1^2,x_1x_2,x_2^2$ come from b).

\end{minipage}
\hspace{0.5 cm}
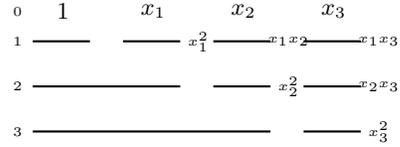
\begin{minipage}{5cm}
  \begin{center}
\begin{tikzpicture}[scale=0.4]
\node at (-0.5,4) [] {${\scriptscriptstyle 0}$};
\node at (-0.5,0) [] {${\scriptscriptstyle 3}$};
\node at (-0.5,1.5) [] {${\scriptscriptstyle 2}$};
\node at (-0.5,3) [] {${\scriptscriptstyle 1}$};
 \draw [thick] (0,0) -- (7.9,0);
 \draw [thick] (9,0) -- (10.9,0);
 \node at (11.5,0) [] {${\scriptscriptstyle
x_3^2}$};
 \draw [thick] (0,1.5) -- (4.9,1.5);
 \draw [thick] (6,1.5) -- (7.9,1.5);
 \node at (8.5,1.5) [] {${\scriptscriptstyle
x_2^2}$};
 \draw [thick] (9,1.5) -- (10.9,1.5);
 \node at (11.5,1.5) [] {${\scriptscriptstyle
x_2x_3}$};
 \draw [thick] (0,3.0) -- (1.9,3.0);
 \draw [thick] (3,3.0) -- (4.9,3.0);
 \node at (5.5,3.0) [] {${\scriptscriptstyle
x_1^2}$};
 \draw [thick] (6,3.0) -- (7.9,3.0);
 \node at (8.5,3.0) [] {${\scriptscriptstyle
x_1x_2}$};
 \draw [thick] (9,3.0) -- (10.9,3.0);
 \node at (11.5,3.0) [] {${\scriptscriptstyle
x_1x_3}$};
 \node at (1,4.0) [] {\small $1$};
 \node at (4,4.0) [] {\small $x_1$};
 \node at (7,4.0) [] {\small $x_2$};
 \node at (10,4.0) [] {\small $x_3$};
\end{tikzpicture}
\end{center}
\end{minipage}

\end{example}

\noindent In \cite{CMR}, given a monomial ideal $I$, the authors define
 the following set, calling it \emph{star set}:
$\mathcal{F}(I)=\left\{x^{\gamma} \in \mathcal{T}\setminus {\sf N}(I) \,
\left\vert \,
\frac{x^{\gamma}}{\min(x^{\gamma})} \right. \in {\sf N}(I) \right\}.$
\begin{Proposition}[\cite{Ce}]\label{DefSt}
With the above notation $\mathcal{F}_{\sf N}=\mathcal{F}(I)$.
\end{Proposition}
The star set $\kF(I)$ of a monomial ideal $I$ is strongly connected to Janet's 
theory \cite{J1,J2,J3,J4} and to the notion of Pommaret basis \cite{Pom, 
PomAk, SeiB}, as explicitly pointed out in \cite{CMR}.
 In particular, for quasi-stable ideals, the star set is finite and coincides with Pommaret bases.
\section{Janet decomposition and completeness.}\label{Jdecsect} 
 Given a monomial/semigroup ideal $J\subset\mathcal{T}$ and its monomial basis ${\sf G}(J)$,
 Janet introduced in  \cite{J1} both the notion of 
 \emph{multiplicative variables} and the connected decomposition of $J$ into disjoint 
  \emph{cones}, characterizing, according to Gerdt-Blinkov notation,
  an \emph{involutive division.}
   \begin{definition}\cite[ppg.75-9]{J1}\label{multiplicative}
Let  $U\subset \mathcal{T}$ be a set of terms
 and  $t=x_1^{\alpha_1}\cdots x_n^{\alpha_n} $
be an element of $U$.
A variable $x_j$ is called \emph{multiplicative}
for $t$ with respect to $U$ if there is no term in
$U$ of the form
$t'=x_1^{\beta_1}\cdots x_j^{\beta_j}x_{j+1}^{\alpha_{j+1}} \cdots 
x_n^{\alpha_n}$
with $\beta_j>\alpha_j$.
 We denote by $M_J(t,U)$ the set of
multiplicative variables for $t$ with respect to $U$.\\
The variables that are not multiplicative for $t$ w.r.t. $U$ 
are called \emph{non-multiplicative} and we denote by $NM_J(t,U)$ the set 
containing them.
\end{definition}
\noindent It is clear that the above definition depends on the order of the variables.
\begin{example}\label{SeveVarOrdPerMolt}
\noindent Consider the set $U=\{x_1,x_2\}\subset\ck[x_1,x_2]$.
If  $x_1<x_2$, then $M_J(x_1,U)=\{x_1\}$, $NM_J(x_1,U)=\{x_2\}$,
$M_J(x_2,U)=\{x_1,x_2\}$, $NM_J(x_2,U)=\emptyset$. If, instead 
 $x_2<x_1$, then $M_J(x_1,U)=\{x_1,x_2\}$, $NM_J(x_1,U)=\emptyset$,
 $M_J(x_2,U)=\{x_2\}$, $NM_J(x_2,U)=\{x_1\}$.
\end{example}
\begin{definition}\label{cone}
With the previous notation, the \emph{cone} of  
$t$ with respect to $U$   is the set
$C_J(t, U):=\{t x_1^{\lambda_1} \cdots x_n^{\lambda_n} \,\vert
\, \textrm{where } \lambda_j\neq 0 \textrm{ only if } x_j 
\textrm{ is multiplicative for }
t \textrm{ w.r.t. } U\}.$
\end{definition}
\begin{example}\label{Molt1}
Consider the set
$J=\{x_1^3,x_2^3, x_1^4x_2x_3,x_3^2\}\subseteq \mathbf{k}[x_1,x_2,x_3]$; suppose 
$x_1<x_2<x_3$. Let $t=x_1^{\alpha_1}x_2^{\alpha_2}x_3^{\alpha_3}=x_1^3$, so $\alpha_1=3,\,\alpha_2=\alpha_3=0$.
The variable $x_1$ is multiplicative for $t$ w.r.t $J$
since there are no terms
 $t'=x_1^{\beta_1}x_2^{\beta_2}x_3^{\beta_3}\in J$ satisfying both 
 conditions  $\beta_1 > 3$and  $\beta_2=\beta_3=0$.
On the other hand, $x_2$ is not multiplicative for $t$ 
since $t''=x_2^3\in U$ satisfies  $t''=x_1^{\gamma_1}x_2^{\gamma_2}
x_3^{\gamma_3}$ with 
$\gamma_2=3>0=\alpha_2$,
$\gamma_3=\alpha_3=0$. Similarly, $x_3$ is not multiplicative since $x_3^2 \in U$. 
In conclusion, we have $M_J(t,U)=\{x_1\}$, $NM_J(t,U)=\{x_2,x_3\}$; 
$C_J(t,U)=\{x_1^h \vert h \in \NN,\, h\geq 3\}$.
\end{example}
\begin{remark}\label{oss: solo tau}  Observe that, 
by definition of multiplicative variable, the only element in 
$C_J(t,U)\cap U$ is $t$ itself. Indeed, if $t \in U$ and also $t s \in  U$ for
a non constant term  $s$, then $\max(s)$ cannot be
multiplicative for $t$, hence.
$t s \notin C_J(t,U)$.
\end{remark}
\noindent Janet introduced then the concept of \emph{complete system} and 
 gave a procedure (\emph{completion})   
  to produce the decomposition in cones.
  \begin{definition}\cite[ppg.75-9]{J1}\label{Complete}
A set of terms $U\subset \mathcal{T}$ is called
\emph{complete} if for every $ t \in U$ and
 $ x_j\in NM_J(t,U)$, there exists $ t' \in U$
such that $x_j t \in C_J(t',U)$.
The term $t'$ is called \emph{involutive divisor} of $x_jt$ w.r.t. Janet division.
  \end{definition}
\noindent  Depending on the notion of multiplicative variable, then also completeness depends
  on the variables' ordering.
  
\begin{remark}\label{singleton}
If $U=\{t\}\subseteq \ck[x_1...,x_n]$ is a singleton, it is complete, 
since  $M_J(t,u)=\{x_1,...,x_n\}$.
\end{remark}
\noindent In the same paper, in order to describe Riquier's \cite{Riq} formulation 
of the description for the general solutions of a PDE problem, Janet 
gave a similar decomposition in terms of disjoint cones, 
generated by multiplicative variables, also for the related normal 
set/order ideal/\emph{escalier}
${\bf{N}}(J)$.\\
The construction of a Bar Code can help to assign to each element $t$ of a finite set of terms
$U\subset \kT$ its multiplicative variables, according to Janet's Definition \ref{multiplicative}.\\
Let $U \subset \kT \subset \ck[x_1,...,x_n] $ be a finite set of terms and suppose
$x_1<x_2<...<x_n$.
As explained in section \ref{BCsect}, we can associate a Bar Code $\cB$ to it. Once $\cB$
is constructed, even if it is not necessary that $\cB$ is an admissible Bar Code,
we can mimick on it the set up we generally perform to construct the star set. In particular:
\begin{itemize}
  \item[a)] $\forall 1 \leq i\leq n$,  place a star symbol $*$ on the right of
  $\cB^{(i)}_{\mu(i)}$;
  \item[b)] $\forall 1 \leq i\leq n-1$, 
  $\forall 1 \leq j \leq \mu(i)-1$ let 
  $\cB^{(i)}_j$ and $\cB^{(i)}_{j+1}$ be two 
  consecutive bars not lying over the 
same $(i+1)$-bar; place a star symbol $*$ between them.
 \end{itemize}
Now, given a term $t \in U$, to detect its multiplicative variables it is enough to check the 
bars over which it lies, as stated in the following proposition (see \cite{BCVJT}).

\begin{Proposition}\label{VMolt}
Let $U \subseteq \kT$ be a finite set of terms and let us denote by $\cB_U$ its Bar Code. For each $t \in U$
$x_i$, $1 \leq i \leq n$ is multiplicative for $t$ if and only if,
in  $\cB_U$, the $i$-bar $\cB^{(i)}_j$, over which $t$ lies, is followed by a star.
\end{Proposition}
\begin{example}\label{BCMolt}
For the set $U=\{x_1^3,x_2^3, x_1^4x_2x_3,x_3^2\}\subseteq
\mathbf{k}[x_1,x_2,x_3]$, $x_1<x_2<x_3$, of example \ref{Molt1}, we have 
the following Bar Code

\begin{minipage}{3.6cm}
\begin{center}
\begin{tikzpicture}[scale=0.8]
\node at (3.6,0.5) [] {${\scriptscriptstyle 0}$};
\node at (3.6,0) [] {${\scriptscriptstyle 1}$};
\node at (3.6,-0.5) [] {${\scriptscriptstyle 2}$};
\node at (4.2,0.5) [] {${\tiny x_1^3}$};
\node at (5.2,0.5) [] {${\tiny x_2^3}$};
\node at (6.2,0.5) [] {${\tiny x_1^4x_2x_3}$};
\node at (7.2,0.5) [] {${\tiny x_3^2}$};

\node at (3.6,-1) [] {${\scriptscriptstyle 3}$};

\node at (4.7,0) [] {${*}$};
\node at (5.7,0) [] {${*}$};
\node at (6.7,0) [] {${*}$};
\node at (7.7,0) [] {${*}$};

\node at (5.7,-0.5) [] {${*}$};
\node at (6.7,-0.5) [] {${*}$};
\node at (7.7,-0.5) [] {${*}$};

\node at (7.7,-1) [] {${*}$};

\draw [thick] (4,0) --(4.5,0);
\draw [thick] (5,0) --(5.5,0);
\draw [thick] (6,0) --(6.5,0);
\draw [thick] (7,0) --(7.5,0);

\draw [thick] (4,-0.5) --(4.5,-0.5);
\draw [thick] (5,-0.5) --(5.5,-0.5);
\draw [thick] (6,-0.5) --(6.5,-0.5);
\draw [thick] (7,-0.5) --(7.5,-0.5);

\draw [thick] (4,-1) --(5.5,-1);
\draw [thick] (6,-1) --(6.5,-1);
\draw [thick] (7,-1) --(7.5,-1);
\end{tikzpicture}
\end{center}
\end{minipage}
\hspace{0.2cm}
\begin{minipage}{8.5cm}
\indent Then, looking at the stars, we can desume that:\\
$\bullet$ $M_J(x_1^3,U)=\{x_1\},\, NM_J(x_1^3,U)=\{x_2,x_3\};$\\
$\bullet$ $M_J(x_2^3,U)=\{x_1x_2\},\, NM_J(x_2^3,U)=\{x_3\} ;$\\
$\bullet$  $M_J(x_1^4x_2x_3,U)\!=\!\{x_1,x_2\}, \!NM_J(x_1^4x_2x_3,U)\!=\!\{x_3\} ;$\\
$\bullet$  $M_J(x_3^2,U)=\{x_1,x_2,x_3\},\,   NM_J(x_3^2,U)=\emptyset.$
\end{minipage}

\end{example}
\begin{remark}\label{SeilerGerdt}
  The Bar Code we are using to detect multiplicative variables is a reformulation of Gerdt-Blinkov-Yanovich \emph{Janet trie} \cite{GBY}, but in the (equivalent) presentation given by Seiler \cite{SeiB}. However, given a finite set of terms, the algorithms for producing its Janet decomposition which can be deduced from both the formulations above of the Janet tree, are different from the algorithm naturally arising from the previous proposition.
\end{remark}

\noindent In \cite{J1}, starting from his definition of multiplicative variable (Definition  
\ref{multiplicative}), Janet deduces the following straightforward corollary, whose proof is 
 reported in \cite{J4}.
\begin{cor}[\cite{J1}]\label{CorollCompl} 
Let $U=\{t_1,...,t_m\}\subseteq \mathcal{T}$ be a finite set of terms, $t_i=x_1^{\alpha_1^{(i)}}\cdots x_n^{\alpha_n^{(i)}}$ and  $t_i'=x_1^{\alpha_1^{(i)}}\cdots 
x_{n-1}^{\alpha_{n-1}^{(i)}}=\frac{t_i}{x_n^{\alpha_n^{(i)}}}$, for $ i=1,...,m$.  Let $U'=\{t'_1,...,t'_m\}$, $\alpha=\max\{\alpha_n^{(i)}$, $1 \leq i \leq m\}$. For each $\lambda \leq \alpha$, we define 
$I_{\lambda}:=\{i :\, 1 \leq i \leq m \vert \alpha_n^{(i)}=\lambda\}$, the 
set indexing the terms in $U$ with $n$-th degree equal to $\lambda$, and
$U'_{\lambda}:=\{t'_i \vert i \in I_{\lambda}\}$.
Then $U$ is complete if and only if the two conditions below hold:
\begin{enumerate}
\item  For each $\lambda \in \{\alpha_n^{(i)},\, 1 \leq i \leq m\} $, $U'_{\lambda}$ is a complete set;
\item $\forall t'_i\in U'_{\lambda}$, $\lambda < \alpha   $,  there exists $j\in \{1,...,m\}$ such that
         \begin{itemize}
          \item $t'_i\in C_J(t'_j,U')$;
           \item  $t'_j\in U'_{\lambda +1}.$
         \end{itemize}
 \end{enumerate}
\end{cor}
\noindent Completeness of a given finite set $U$ can be detected by exploiting the 
Bar Code, as stated in the following proposition.
\begin{Proposition}\label{CompletoTest}
  Let $U \subseteq \kT$ be a finite set of terms and $\cB$ be its Bar Code.
  Let $t \in U$, $x_i \in NM_J(t,U)$ and  $\cB^{(i)}_j$ the $i$-bar under 
  $t$.\\
  Let $s \in U$; it holds $s \mid_J x_it$ if and only if
  \begin{enumerate}
   \item $s \mid x_it$
   \item $s$ lies over  $\cB^{(i)}_{j+1}$ and 
   \item $\forall j'$ appearing with nonzero 
  exponent in $\frac{x_it}{s}$ there is a star after the 
$j'$-bar under $s$. 
  \end{enumerate}
\end{Proposition}
\begin{proof}
  ``$\Leftarrow$'' It is an obvious consequence of Proposition \ref{VMolt}; indeed, by 1. $s \mid x_it$.  Thanks to 3.,
  all the variables in $\frac{x_it}{s}$ are multiplicative. Note that $x_i$ is not a variable of  $w:=\frac{x_it}{s}$ and 
 it does not need to be multiplicative for $s$, since, $s$ lies over  $\cB^{(i)}_{j+1}$, so $\deg_i(s)=\deg_i(t)+1$.
 So $sw=x_it$ and $w$ contains only multiplicative variables for $s$; therefore $s \mid_J x_it$.
  \\
  ``$\Rightarrow$''  Let $s \in U$,  $s \mid_J x_it$; $s\mid x_it$ by definition of Janet division\footnote{And actually of involutive division.}.\\
  If $s$ would lie over  $\cB^{(i)}_j$, then $\deg_l(s)=\deg_l(t)$ for $l=i,...,n$, i.e. in $s$ and $t$ the variables $x_i,...,x_n$ appear with the same exponent. Then, being $s \mid x_it$, $x_i \in 
V_s:=\{x_j, 1\leq j\leq n: \,  x_j \mid w:=\frac{x_it}{s}\}$, so $x_i$ should be multiplicative for $s$, this meaning having a star after $\cB^{(i)}_j$, which is impossible by hypothesis, since 
in this case $x_it \in C_J(s,U) \cap C_J(t,U)$.\\
  If $s$ would lie over  $\cB^{(i)}_l$, $l >j+1$, there exists $h\in \{i,...,n\}$ s.t. $\deg_h(s)>\deg_h(x_it)$, so $s \nmid x_it$, which is again a contradiction.\\
  If $s$ would lie over  $\cB^{(i)}_l$, $l <j$, then $s<_{Lex} t$ and it cannot happen that $\deg_{l'}(s)=\deg_{l'}(t)$ for $l'=i,...,n$ (since otherwise $s$ would have been over $\cB^{(i)}_j$). Let 
$x_k:=\max\{x_h,\, h=1,...,n \vert \deg_h(s)<\deg_h(t)\}$; then, since $t \in U$ and $\deg_n(t)=\deg_n(s),...,$ $\deg_{k+1}(t)=\deg_{k+1}(s)$ and $\deg_k(t)>\deg_k(s)$, by definition of multiplicative 
variable according to  Janet division, $x_k \in NM_J(s,U)$.
  Then $s$ must lie over $\cB^{(i)}_{j+1}$.
  For being $s \mid x_it$, all the variables appearing with nonzero exponent in $\frac{x_it}{s}$ must be multiplicative for $s$, and this implies that 
   $\forall j'$ appearing with nonzero exponent in $\frac{x_it}{s}$ there is a star after the $j'$-bar under $s$, by Proposition \ref{VMolt}. 
\end{proof}

\noindent From Proposition \ref{CompletoTest} we finally get the following
\begin{thrm}\label{CriterioCompleto}
  Let $U \subseteq \kT$ be a finite set of terms and $\cB$ be its Bar Code. Then $U$ is a complete set if and only if $\forall t \in U$, 
 $\forall x_i \in NM_J(t,U)$, called  $\cB^{(i)}_j$ the $i$-bar under 
  $t$, there exists a term $s \in U$  satisfying conditions $1,2,3$ of Proposition \ref{CompletoTest}.
\end{thrm}

\noindent According to  Proposition \ref{CompletoTest} and Theorem \ref{CriterioCompleto}, given a finite set of term $U \subseteq \kT$,
to check its completeness we take, $\forall t \in U$, $\forall x_i \in NM_J(t,U)$, the
$i$-bar $\cB_j^{(i)},\, 1\leq j \leq \mu(i)$ under $t$ and we look for an involutive divisor among
 the terms over $\cB_{j+1}^{(i)}$, checking conditions 1,3 above. We see now two simple 
 examples of this procedure.
\begin{example}\label{esempioNonCompleto}
 Coming back to examples \ref{Molt1} and  \ref{BCMolt}, we consider again 
 $U=\{x_1^3,x_2^3, x_1^4x_2x_3,$  $x_3^2\}\subseteq
\mathbf{k}[x_1,x_2,x_3]$ and its Bar Code

\begin{minipage}{4.5cm}
\begin{center}
\begin{tikzpicture}
\node at (3.6,0.5) [] {${\scriptscriptstyle 0}$};
\node at (3.6,0) [] {${\scriptscriptstyle 1}$};
\node at (3.6,-0.5) [] {${\scriptscriptstyle 2}$};
\node at (4.2,0.5) [] {${\small x_1^3}$};
\node at (5.2,0.5) [] {${\small x_2^3}$};
\node at (6.2,0.5) [] {${\small x_1^4x_2x_3}$};
\node at (7.2,0.5) [] {${\small x_3^2}$};

\node at (3.6,-1) [] {${\scriptscriptstyle 3}$};

\node at (4.7,0) [] {${*}$};
\node at (5.7,0) [] {${*}$};
\node at (6.7,0) [] {${*}$};
\node at (7.7,0) [] {${*}$};

\node at (5.7,-0.5) [] {${*}$};
\node at (6.7,-0.5) [] {${*}$};
\node at (7.7,-0.5) [] {${*}$};

\node at (7.7,-1) [] {${*}$};

\draw [thick] (4,0) --(4.5,0);
\draw [thick] (5,0) --(5.5,0);
\draw [thick] (6,0) --(6.5,0);
\draw [thick] (7,0) --(7.5,0);

\draw [thick] (4,-0.5) --(4.5,-0.5);
\draw [thick] (5,-0.5) --(5.5,-0.5);
\draw [thick] (6,-0.5) --(6.5,-0.5);
\draw [thick] (7,-0.5) --(7.5,-0.5);

\draw [thick] (4,-1) --(5.5,-1);
\draw [thick] (6,-1) --(6.5,-1);
\draw [thick] (7,-1) --(7.5,-1);
\end{tikzpicture}
\end{center}
\end{minipage}
\hspace{0.1cm}
\begin{minipage}{7cm}
Take $t=x_1^3$ and $x_2 \in NM_J(t,U)=\{x_2,x_3\}$; $t$ lies over $\cB^{(2)}_1$
and the only term over $\cB^{(2)}_2$ is $x_2^3 \nmid x_1^3x_2=tx_2$, so $tx_2$ 
has no involutive divisor on $U$ and this implies that our set is actually non-complete.
\end{minipage}

\end{example}
\begin{example}\label{esempioCompleto} 
 Consider the set $U=\{x^2,xy\}\subset \ck[x,y],\, x<y$. Its Bar Code is 
 
\begin{minipage}{3cm} 
\begin{center}
\begin{tikzpicture}
\node at (3.6,0.5) [] {${\scriptscriptstyle 0}$};
\node at (3.6,0) [] {${\scriptscriptstyle 1}$};
\node at (3.6,-0.5) [] {${\scriptscriptstyle 2}$};
\node at (4.2,0.5) [] {${\small x^2}$};
\node at (5.2,0.5) [] {${\small xy}$};

\node at (4.7,0) [] {${*}$};
\node at (5.7,0) [] {${*}$};

\node at (5.7,-0.5) [] {${*}$};

\draw [thick] (4,0) --(4.5,0);
\draw [thick] (5,0) --(5.5,0);
 
\draw [thick] (4,-0.5) --(4.5,-0.5);
\draw [thick] (5,-0.5) --(5.5,-0.5);

\end{tikzpicture}
\end{center}
\end{minipage}
\begin{minipage}{9cm}
so, looking at the stars, we can desume $M_J(x^2,U)=\{x\}$, $NM_J(x^2,U)=\{y\}$, $M_J(xy,U)=\{x,y\}$, $NM_J(xy,U)=\emptyset$.
Now, $t=x^2$ lies over $\cB^{(2)}_1$ and over 
 $\cB^{(2)}_2$ there is only $xy$ 
\end{minipage}

 s.t. $xy \mid x^2y$.
 Since $x \in M_J(xy,U)$, $xy \mid_J x^2y$ and we can conclude that $U$ is complete, w.r.t. the 
 given ordering on the variables.
 \end{example}
\section{A greedy algorithm for complete sets.}\label{greedysect}

In this section, given a finite set of terms $U=\{t_1,...,t_m\}\subseteq \mathcal{T}$,
we try to find out whether there exists an
  \emph{ordering on the variables} $x_1,...,x_n$ such that $U$ 
is \emph{complete}.
As explained in section \ref{Jdecsect}, the Bar Code allows to detect the completeness of $U$. 
Clearly such a construction depends on the variables' ordering, 
so if we want to solve the problem, in principle,  we 
should draw and check $n!$ different Bar Codes, which turns out to
be rather tedious and time consuming.
Exploiting again the Bar Code  and Corollary \ref{CorollCompl},
we can look for the solution of our problem in a ``greedy" way, 
so that most of the tests can be skipped.
The idea consists in constructing the Bar Code $\cB$ of the set $U=\{t_1,...,t_m\}\subset \kT$
from the maximal variable to the minimal one, checking if, with the choice 
made up to the current point on the variables' ordering, the conditions of 
 Proposition \ref{CompletoTest} hold for each term in $U$, and going back retracting our steps
 in case of failure, so modifying previous choices.
\\ 
Let $X=\{x_1,...,x_n\}$ be the set of all variables. 
In the first step we look for the subset $Y\subseteq X$ of  good candidates for being the maximal variable, 
scanning  the elements of $X$ (see Corollary \ref{CorollCompl}).
For $i=1,...,n$, we compute the sets 
$D_i:=\{\beta \in \NN \vert \exists t \in U,\, \deg_{i}(t)=\beta\}$.
If for some $\gamma\in D_i$ $\gamma <\max(D_i)$, $\gamma + 1 \notin D_i$, then 
$x_i$ cannot be the maximal variable. Indeed, if $x_i$ would be the maximal variable,
then 
 there would exist $t_1'=\frac{t_1}{x_1^{\gamma}}\in U_{\gamma}'$
 and by Corollary \ref{CorollCompl}, we would need a term 
in $U_{\gamma+1}'$, which is actually the empty set.
Clearly, if $Y=\emptyset$, no variable is suitable for being the maximal one, making $U$
complete and this implies that $U$ is not complete for any variables' ordering.
\begin{example}\label{Salti}
Consider $U=\{x_1x_2^3,x_1^3x_2\}\subset \mathbf{k}[x_1,x_2]$. Such a set is not
complete  since $D_1=D_2=\{1,3\}$. As a confirmation, we can see that, if  $x_1<x_2$, we have

\begin{minipage}{4cm}
\begin{center}
\begin{tikzpicture}
\node at (3.6,0.5) [] {${\scriptscriptstyle 0}$};
\node at (3.6,0) [] {${\scriptscriptstyle 1}$};
\node at (3.6,-0.5) [] {${\scriptscriptstyle 2}$};

\node at (4.2,0.5) [] {${\small x_1^3x_2}$};
\node at (5.2,0.5) [] {${\small x_1x_2^3}$};

\node at (4.7,0) [] {${*}$};
\node at (5.7,0) [] {${*}$};

\node at (5.7,-0.5) [] {${*}$};

\draw [thick] (4,0) --(4.5,0);
\draw [thick] (5,0) --(5.5,0);

\draw [thick] (4,-0.5) --(4.5,-0.5);
\draw [thick] (5,-0.5) --(5.5,-0.5);

\end{tikzpicture}
\end{center}
\end{minipage}
\hspace{0.1cm}
\begin{minipage}{8cm}
Then $M_J(x_1^3x_2,U)=\{x_1	\}$,  $M_J(x_1x_2^3,U)=\{x_1,x_2	\}$ and 
$x_1^3x_2^2$ does not belong neither to $C_J(x_1^3x_2,U)$ nor to $C_J(x_1x_2^3,U).$
\end{minipage}

On the other hand, if $x_2<x_1$, we have

\begin{minipage}{4cm}
\begin{center}
\begin{tikzpicture}
\node at (3.6,0.5) [] {${\scriptscriptstyle 0}$};
\node at (3.6,0) [] {${\scriptscriptstyle 2}$};
\node at (3.6,-0.5) [] {${\scriptscriptstyle 1}$};
\node at (4.2,0.5) [] {${\small x_1x_2^3}$};
\node at (5.2,0.5) [] {${\small x_1^3x_2}$};

\node at (4.7,0) [] {${*}$};
\node at (5.7,0) [] {${*}$};

\node at (5.7,-0.5) [] {${*}$};

\draw [thick] (4,0) --(4.5,0);
\draw [thick] (5,0) --(5.5,0);

\draw [thick] (4,-0.5) --(4.5,-0.5);
\draw [thick] (5,-0.5) --(5.5,-0.5);

\end{tikzpicture}
\end{center}\end{minipage}
\hspace{0.1cm}
\begin{minipage}{8cm}
Thus $M_J(x_1x_2^3,U)=\{x_2\}$,  $M_J(x_1^3x_2,U)=\{x_1,x_2\}$ and $x_1^2x_2^3$ 
does not belong neither to $C_J(x_1^3x_2,U)$ nor to $C_J(x_1x_2^3,U).$\end{minipage}

\end{example}
Suppose now $\emptyset \neq Y =\{x_{j_1},...,x_{j_l}\}\subseteq X$; we start picking 
$x_{j_1}\in Y$ as maximal variable. We reorder the elements of $U$, increasingly w.r.t.
their $j_1$-degree; we pose the only condition $t<t'$ when $t\mid t'$
 for some $t,t' \in U$ with $deg_{j_1}(t)=deg_{j_1}(t')$.
Then, we write the corresponding $j_1$-bars $\cB^{(j_1)}_1,...,\cB^{(j_1)}_{\mu(j_1)}$ under the terms (one bar under
 each terms of some $j_1$-degree).
\begin{example}\label{running1}
Let us consider the set $U=\{x_1,x_1^2,x_2,x_1x_3\}\subset \ck[x_1,x_2,x_3]$; we first compute 
$D_1=\{0,1,2\}$, $D_2=D_2=\{0,1\}$. All the variables are good candidates for
being the maximal one. We pick, for example, $x_3$, so we have 
\begin{center}
\begin{tikzpicture}
\node at (3.6,0.5) [] {${\scriptscriptstyle 0}$};

\node at (3.6,-1) [] {${\scriptscriptstyle 3}$};
\node at (4.2,0.5) [] {${\small x_1}$};
\node at (5.2,0.5) [] {${\small x_1^2}$};
\node at (6.2,0.5) [] {${\small x_2}$};
\node at (7.2,0.5) [] {${\small x_1x_3}$};

\draw [thick] (4,-1) --(6.5,-1);
\draw [thick] (7,-1) --(7.5,-1);
\node at (7.7,-1) [] {${*}$};
\end{tikzpicture}
\end{center}
We remark that, we could also have picked another variables, obtaining a different Bar Code,
for example, picking $x_1$, we would have got:

\begin{center}
\begin{tikzpicture}
\node at (3.6,0.5) [] {${\scriptscriptstyle 0}$};

\node at (3.6,-1) [] {${\scriptscriptstyle 1}$};
\node at (4.2,0.5) [] {${\small x_2}$};
\node at (5.2,0.5) [] {${\small x_1}$};
\node at (6.2,0.5) [] {${\small x_1x_3}$};
\node at (7.2,0.5) [] {${\small x_1^2}$};

\draw [thick] (4,-1) --(4.5,-1);
\draw [thick] (5,-1) --(6.5,-1);
\draw [thick] (7,-1) --(7.5,-1);
\node at (7.7,-1) [] {${*}$};
\end{tikzpicture}
\end{center}
\end{example}
\noindent Now, with the routine Friends, we look for candidate terms for having 
condition $2$ of Corollary \ref{CorollCompl} satisfied, so for $1 \leq j < \mu(j_1)$, we fix 
$\cB^{(j_1)}_j$ and, for each $t$ over $\cB^{(j_1)}_j$ we define the set 
 $U(t,x_j)=\{(u,\alpha) \vert u \in B^{(j_1)}_{j+1} \textrm{ and } \alpha : 
 tx_j=um,\, m \in \kT[\alpha]\}$ of the candidate involutive divisors for 
 $tx_{j_1}$ (notice that $x_{j_1}$ is the maximal variable, so 
 $x_{j_1}\in M_J(v,U) \Leftrightarrow v$ lies over $\cB^{(j_1)}_{\mu(j_1)}$ by Proposition
 \ref{VMolt}).
 If exists $1 \leq j < \mu(j_1)$ and $\exists t' $ over $\cB^{(j_1)}_j$ with 
 $U(t',x_{j_1})=\emptyset$, then $x_{j_1}$ is not a good candidate for being the maximal variable,
so we come back to $Y$ and we start again with a new maximal variable.
\begin{example}\label{running2}
Coming back to example \ref{running1}, we have 
$U(x_1,x_3)=\{(x_1x_3,\emptyset)\}$, $U(x_1^2,x_3)=\{(x_1x_3,\{x_1\})\}$ and 
$U(x_2,x_3)=\emptyset$,
so $x_3$ was a bad choice for being the maximal variable and we try with $x_1$,
getting 

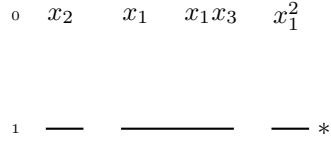
\begin{center}
\begin{tikzpicture}
\node at (3.6,0.5) [] {${\scriptscriptstyle 0}$};
\node at (3.6,-1) [] {${\scriptscriptstyle 1}$};
\node at (4.2,0.5) [] {${\small x_2}$};
\node at (5.2,0.5) [] {${\small x_1}$};
\node at (6.2,0.5) [] {${\small x_1x_3}$};
\node at (7.2,0.5) [] {${\small x_1^2}$};

\draw [thick] (4,-1) --(4.5,-1);
\draw [thick] (5,-1) --(6.5,-1);
\draw [thick] (7,-1) --(7.5,-1);
\node at (7.7,-1) [] {${*}$};
\end{tikzpicture}
\end{center}
Now, $U(x_2,x_1)=\{(x_1,\{x_2\})\}$, 
$U(x_1,x_1)=\{(x_1^2,\emptyset)\}$
and $U(x_1x_3,x_1)=\{(x_1^2, \{x_3\})\}$,
so $x_1$, at least for now, is a good choice for the maximal variable.
\end{example}
Suppose to be in the non-failure case; if for  $1 \leq j \leq \mu(j_1)$
there is only one term over $\cB^{(j_1)}_j$, all the bars are \emph{unitary} so we say 
that we are in the \emph{unitary case}. In this case, each variable ordering s.t. $x_{j_1}$
is the maximal variable makes $U$ a complete set of terms.
Indeed, in this case, for each choice of the following 
variables, their corresponding bars will be unitary again and, by the construction 
of the stars, all of them will be followed by a star. In other words, 
for each $t \in U$, and for each $x_j \neq x_{j_1}$, $x_j \in M_J(t,U)$. Moreover
$\forall t \in U$, $\vert U(t,x_{j_1})\vert =1$, so let $(u, \alpha)$ be the only element
of $U(t,x_{j_1})$, then all variables in $\alpha$ differ from $x_{j_1}$, 
so they are multiplicative for 
$u$ and this makes $u$ the required involutive divisor of $x_{j_1}t$, ensuring the completeness 
of $U$.
\begin{example}\label{unitary}
For $U=\{x^3,xy,y^2\} \subset \ck[x,y]$, $D_1=\{0,1,3\}$ and $D_2=\{0,1,2\}$, so $Y=\{y\}$.
Chosing $y$ as maximal variable we have:

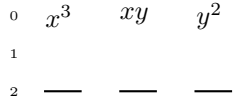
\begin{center}
\begin{tikzpicture}
\node at (3.6,0.5) [] {${\scriptscriptstyle 0}$};
\node at (3.6,0) [] {${\scriptscriptstyle 1}$};
\node at (3.6,-0.5) [] {${\scriptscriptstyle 2}$};
\node at (4.2,0.5) [] {${\small x^3}$};
\node at (5.2,0.5) [] {${\small xy}$};
\node at (6.2,0.5) [] {${\small y^2}$};

\draw [thick] (4,-0.5) --(4.5,-0.5);
\draw [thick] (5,-0.5) --(5.5,-0.5);
\draw [thick] (6,-0.5) --(6.5,-0.5);

\end{tikzpicture}
\end{center}
All the $2$-bars are unitary, so, completing the Bar Code we get 

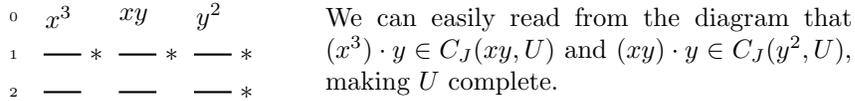
\begin{minipage}{4.5cm}
\begin{center}
\begin{tikzpicture}
\node at (3.6,0.5) [] {${\scriptscriptstyle 0}$};
\node at (3.6,0) [] {${\scriptscriptstyle 1}$};
\node at (3.6,-0.5) [] {${\scriptscriptstyle 2}$};
\node at (4.2,0.5) [] {${\small x^3}$};
\node at (5.2,0.5) [] {${\small xy}$};
\node at (6.2,0.5) [] {${\small y^2}$};
\node at (4.7,0) [] {${*}$};
\node at (5.7,0) [] {${*}$};
\node at (6.7,0) [] {${*}$};
\node at (6.7,-0.5) [] {${*}$};

\draw [thick] (4,0) --(4.5,0);
\draw [thick] (5,0) --(5.5,0);
\draw [thick] (6,0) --(6.5,0);

\draw [thick] (4,-0.5) --(4.5,-0.5);
\draw [thick] (5,-0.5) --(5.5,-0.5);
\draw [thick] (6,-0.5) --(6.5,-0.5);

\end{tikzpicture}
\end{center}
\end{minipage}
\hspace{0.1cm}
\begin{minipage}{7cm}

We can easily read from the
diagram that $(x^3)\cdot y \in C_J(xy,U)$ and $(xy)\cdot y \in C_J(y^2,U)$, 
making $U$ complete.\end{minipage}

\end{example}
 
If we are not in the unitary case, we have to choose the next variable and continue drawing
the Bar Code, using the routine Common.
\\
To get the candidates for being the next variable, we execute the procedure CandidateVar 
to each $j_1$-bar and (procedure Candidate) we intersect the results. If the intersection is empty 
then $x_{j_1}$ was not a good choice for being the maximal variable and we have to come back
and repeat the whole procedure for another maximal variable.
Otherwise, we choose some $x_{j_2}$ among the variables in the intersection, and for each 
$1 \leq j \leq \mu(j_1)$, we order the terms over $\cB^{(j_1)}_j$ exactly as 
done for constructing the $j_1$-bars and we draw all the $j_2$-bars.
Employing again the routine Friends, separately for each $j_1$-bar,
we look for candidate involutive divisors when $x_{j_2}$ is not multiplicative.
Moreover, we check whether the choice of $x_{j_2}$ is suitable to the candidates
found in the previous step. Indeed, for each $t$ over $\cB^{(j_1)}_j$, $1 \leq j< \mu(j_1)$,
we have constructed a set of candidates $U(t,x_{j_1})$. Given 
$(u, \alpha) \in U(t,x_{j_1})$, if $x_{j_2} \notin \alpha$, then the multiplicativity of 
$x_{j_2}$ is irrelevant for $u$, so $(u,\alpha)$ still remains a good candidate 
for being an involutive divisor. It is still a good candidate also if  
$x_{j_2} \in \alpha$ and the $j_2$-bar of $u$ is in one of the conditions for being followed 
by a star (see section \ref{BCsect}), since it means that $x_{j_2} $ is multiplicative for $u$.
Otherwise we remove $(u, \alpha)$ from the candidates. If for some $t$ its candidate list is empty
we have to revoke the choice of $x_{j_2}$ and come back with another candidate.
\\    
If the procedure gives a positive outcome, then a new variable has been chosen 
and the routine Common keeps calling itself until
\begin{itemize}
 \item  all variables have been placed (positive outcome)
 \item the unitary case is reached (positive outcome)
 \item  continue revocations of choices lead to failure (negative outcome).
\end{itemize}
\begin{example}\label{running3}
 We conclude now examples \ref{running1} and \ref{running2}. From 
 \begin{center}
\begin{tikzpicture}
\node at (3.6,0.5) [] {${\scriptscriptstyle 0}$};
\node at (3.6,-1) [] {${\scriptscriptstyle 1}$};
\node at (4.2,0.5) [] {${\small x_2}$};
\node at (5.2,0.5) [] {${\small x_1}$};
\node at (6.2,0.5) [] {${\small x_1x_3}$};
\node at (7.2,0.5) [] {${\small x_1^2}$};

\draw [thick] (4,-1) --(4.5,-1);
\draw [thick] (5,-1) --(6.5,-1);
\draw [thick] (7,-1) --(7.5,-1);
\node at (7.7,-1) [] {${*}$};
\end{tikzpicture}
\end{center}
we choose now $x_2$ as following variable and we get 
\begin{center}
\begin{tikzpicture}
\node at (3.6,0.5) [] {${\scriptscriptstyle 0}$};
 \node at (3.6,-0.5) [] {${\scriptscriptstyle 2}$};
\node at (3.6,-1) [] {${\scriptscriptstyle 1}$};
\node at (4.2,0.5) [] {${\small x_2}$};
\node at (5.2,0.5) [] {${\small x_1}$};
\node at (6.2,0.5) [] {${\small x_1x_3}$};
\node at (7.2,0.5) [] {${\small x_1^2}$};

\draw [thick] (4,-0.5) --(4.5,-0.5);
\draw [thick] (5,-0.5) --(6.5,-0.5);
\draw [thick] (7,-0.5) --(7.5,-0.5);
\node at (4.7,-0.5) [] {${*}$};
\node at (6.7,-0.5) [] {${*}$};

\node at (7.7,-0.5) [] {${*}$};

\draw [thick] (4,-1) --(4.5,-1);
\draw [thick] (5,-1) --(6.5,-1);
\draw [thick] (7,-1) --(7.5,-1);
\node at (7.7,-1) [] {${*}$};
\end{tikzpicture}
\end{center}
Since $x_2$ is multiplicative for all terms, and since 
over each $x_1$-bar there is only one $x_2$-bar, Friends gives a positive 
outcome.  Finally choosing $x_3$, we get 

\begin{center}
\begin{tikzpicture}
\node at (3.6,0.5) [] {${\scriptscriptstyle 0}$};
 \node at (3.6,0) [] {${\scriptscriptstyle 3}$};
\node at (3.6,-0.5) [] {${\scriptscriptstyle 2}$};
\node at (3.6,-1) [] {${\scriptscriptstyle 1}$};
\node at (4.2,0.5) [] {${\small x_2}$};
\node at (5.2,0.5) [] {${\small x_1}$};
\node at (6.2,0.5) [] {${\small x_1x_3}$};
\node at (7.2,0.5) [] {${\small x_1^2}$};
\draw [thick] (4,0) --(4.5,0);
\draw [thick] (5,0) --(5.5,0);
\draw [thick] (6,0) --(6.5,0);
\draw [thick] (7,0) --(7.5,0);
\node at (4.7,0) [] {${*}$};
\node at (6.7,0) [] {${*}$};

\node at (7.7,0) [] {${*}$};

\draw [thick] (4,-0.5) --(4.5,-0.5);
\draw [thick] (5,-0.5) --(6.5,-0.5);
\draw [thick] (7,-0.5) --(7.5,-0.5);
\node at (4.7,-0.5) [] {${*}$};
\node at (6.7,-0.5) [] {${*}$};

\node at (7.7,-0.5) [] {${*}$};

\draw [thick] (4,-1) --(4.5,-1);
\draw [thick] (5,-1) --(6.5,-1);
\draw [thick] (7,-1) --(7.5,-1);
\node at (7.7,-1) [] {${*}$};
\end{tikzpicture}
\end{center}
Now, $x_3$ is multiplicative for $x_1^2$ as required by $U(x_1x_3,x_1)$ and 
we have $U(x_1,x_3)=\{(x_1x_3,\emptyset)\}$, so $U$ turns out to be complete with 
the variables' ordering $x_3<x_2<x_1$.
\\
We point out that this is not the only ordering making $U$ complete, in particular, for 
$x_1<x_3<x_2$ $U$ is complete again:

\begin{center}
\begin{tikzpicture}
\node at (3.6,0.5) [] {${\scriptscriptstyle 0}$};
 \node at (3.6,0) [] {${\scriptscriptstyle 1}$};
\node at (3.6,-0.5) [] {${\scriptscriptstyle 3}$};
\node at (3.6,-1) [] {${\scriptscriptstyle 2}$};
\node at (4.2,0.5) [] {${\small x_1}$};
\node at (5.2,0.5) [] {${\small x_1^2}$};
\node at (6.2,0.5) [] {${\small x_1x_3}$};
\node at (7.2,0.5) [] {${\small x_2}$};
\draw [thick] (4,0) --(4.5,0);
\draw [thick] (5,0) --(5.5,0);
\draw [thick] (6,0) --(6.5,0);
\draw [thick] (7,0) --(7.5,0);
\node at (5.7,0) [] {${*}$};
\node at (6.7,0) [] {${*}$};

\node at (7.7,0) [] {${*}$};

\draw [thick] (4,-0.5) --(5.5,-0.5);
\draw [thick] (6,-0.5) --(6.5,-0.5);
\draw [thick] (7,-0.5) --(7.5,-0.5);
\node at (6.7,-0.5) [] {${*}$};

\node at (7.7,-0.5) [] {${*}$};

\draw [thick] (4,-1) --(6.5,-1);
\draw [thick] (7,-1) --(7.5,-1);
\node at (7.7,-1) [] {${*}$};
\end{tikzpicture}
\end{center}
Indeed 
\begin{itemize}
 \item $M_J(x_1,U)=\emptyset$, $NM_J(x_1,U)=\{x_1,x_2,x_3\}$, with
 $x_1^2\in C_J(x_1^2,U)$, $x_1x_2 \in C_J(x_2,U)$, $x_1x_3 \in C_J(x_1x_3,U)$ ;
 \item $M_J(x_1^2,U)=\{x_1\}$, $NM_J(x_1^2,U)=\{x_2,x_3\}$, 
with $x_1^2x_2 \in C_J(x_2,U)$, $x_1^2x_3 \in C_J(x_1x_3,U) $;
\item $M_J(x_1x_3,U)=\{x_1,x_3\}$, $NM_J(x_1x_3,U)=\{x_2\}$, with $x_1x_2x_3 \in C_J(x_2,U)$; 
 \item $M_J(x_2,U)=\{x_1,x_2,x_3\}$, $NM_J(x_2,U)=\emptyset$.
\end{itemize}
\end{example}

The pseudocode of all mentioned routines is displayed in Appendix \ref{Codice}. We see now a complete example for the whole procedure. 
\begin{example}\label{Esempione}
Consider the set 
$$M=\{x_2x_3,x_1^2,x_3^2, 
x_2^2,x_1x_2,x_1x_2x_4,x_1^2x_4,x_4x_3,x_2^2x_4,x_1^2x_3\}\subset 
\mathbf{k}[x_1,x_2,x_3,x_4].$$
First, we compute $D_1=D_2=D_3=\{0,1,2\}$, $D_4=\{0,1\}$, desuming that 
each variable is a good candidate for being the maximal one, so $Y=\{x_1,x_2,x_3,x_4\}$.
We choose, for example, $x_3$, getting
\begin{center}
\begin{tikzpicture}
\node at (4.2,0) [] {${\small x_1^2}$};
\node at (5.2,0) [] {${\small x_1x_2}$};
\node at (6.2,0) [] {${\small x_2^2}$};
\node at (7.2,0) [] {${\small x_1^2x_4}$};
\node at (8.2,0) [] {${\small x_1x_2x_4}$};
\node at (9.2,0) [] {${\small x_2^2x_4}$};
\node at (10.2,0) [] {${\small x_1^2x_3}$};
\node at (11.2,0) [] {${\small x_2x_3}$};
\node at (12.2,0) [] {${\small x_4x_3}$};
\node at (13.2,0) [] {${\small x_3^2}$};

\node at (3.6,-2) [] {${\scriptscriptstyle 3}$};
\draw [thick] (4,-2) --(9.5,-2);
\draw [thick] (10,-2) --(12.5,-2);
\draw [thick] (13,-2) --(13.5,-2);
\node at (13.7,-2) [] {${*}$};

\end{tikzpicture}
\end{center}
Now, running Friends for the first time, we get

\begin{minipage}{8.5cm}
$\bullet$ $U(x_1^2, x_3)=\{(x_1^2x_3,\emptyset)\}$;\\
$\bullet$ $U(x_1x_2, x_3 )=\{(x_2x_3,\{x_1\})\}$;\\
$\bullet$ $U(x_2^2, x_3 )=\{(x_2x_3,\{x_2\})\}$;\\
 $\bullet$ $U(x_1^2x_4, x_3 )=\{( x_1^2x_3,\{x_4\}),( x_3x_4,\{x_1\})\}$;\\  
$\bullet$ $U(x_1x_2x_4, x_3)=\{( x_2x_3,\{x_1,x_4\}),(x_3x_4,\{x_1,x_2\})\}$;\\
$\bullet$ $U(x_2^2x_4, x_3 )=\{(x_2x_3,\{x_2,x_4\}),( x_3x_4,\{x_2\})\}$;
 \end{minipage}
\hspace{0.1cm}
\begin{minipage}{6cm}
$\bullet$ $U(x_1^2x_3, x_3 )=\{(x_3^2,\{x_1 \})\}$;\\
 $\bullet$ $U(x_2x_3, x_3   )=\{(x_3^2,\{ x_2\})\}$;\\
$\bullet$ $U(x_3x_4, x_3 )=\{(x_3^2,\{x_4\})\}$.
\end{minipage}

The procedure gives a positive outcome, so, since we are not in the unitary case, we 
apply Common.  All the variables are good candidates for being the second in order of magnitude 
and, for example, we choose $x_4$, getting:
\begin{center}
\begin{tikzpicture}

\node at (4.2,0) [] {${\small x_1^2}$};
\node at (5.2,0) [] {${\small x_1x_2}$};
\node at (6.2,0) [] {${\small x_2^2}$};
\node at (7.2,0) [] {${\small x_1^2x_4}$};
\node at (8.2,0) [] {${\small x_1x_2x_4}$};
\node at (9.2,0) [] {${\small x_2^2x_4}$};
\node at (10.2,0) [] {${\small x_1^2x_3}$};
\node at (11.2,0) [] {${\small x_2x_3}$};
\node at (12.2,0) [] {${\small x_3x_4}$};
\node at (13.2,0) [] {${\small x_3^2}$};

 \draw [thick] (4,-1.5)--(6.5,-1.5);
 \draw [thick] (7,-1.5)--(9.5,-1.5);
\draw [thick] (10,-1.5)--(11.5,-1.5);
\draw [thick] (12,-1.5)--(12.5,-1.5);
\draw [thick] (13,-1.5)--(13.5,-1.5);

\draw [thick] (4,-2) --(9.5,-2);
\draw [thick] (10,-2) --(12.5,-2);
\draw [thick] (13,-2) --(13.5,-2);

\node at (3.6,-1.5) [] {${\scriptscriptstyle 4}$};
\node at (3.6,-2) [] {${\scriptscriptstyle 3}$};

\node at (9.7,-1.5) [] {${*}$};
\node at (12.7,-1.5) [] {${*}$};
\node at (13.7,-1.5) [] {${*}$};

\node at (13.7,-2) [] {${*}$};

\end{tikzpicture}
\end{center}
We have:

\begin{minipage}{7cm}
\begin{itemize}
 \item $U(x_1^2,x_4)=\{(x_1^2x_4,\emptyset),(x_4, \{x_1\})\}$;
 \item $U(x_1x_2,x_4)=\{(x_1x_2x_4,\emptyset)\}$;
 \end{itemize}
 \end{minipage}
 \hspace{0.1cm}
 \begin{minipage}{6cm}
 \begin{itemize}
   \item $U(x_2^2,x_4)=\{(x_2^2x_4,\emptyset)\}$;
 \item $U(x_2,x_4)=\{(x_4, \{x_2\})\}$.
\end{itemize} \end{minipage}

We check that the choice of $x_4$ is suitable for the condition imposed 
in the previous step
\begin{itemize}
 \item for $U(x_1^2x_4, x_3 )=\{( x_1^2x_3,\{x_4\}),( x_3x_4,\{x_1\})\}$
 e notice that $x_1^2x_3$ does not lie on the rightmost $4$-bar, so $x_4$ is 
not multiplicative. Since we have more than one term associated to $x_1^2x_4$, 
we only delete $x_1^2x_3$ and keep $x_3x_4.$
\\
The same argument holds for  $x_1x_2x_4,x_2^2x_4$.
 \item For $U(x_3x_4, x_3 )=\{(x_3^2,\{x_4\})\}$, since  $x_3^2$ lies on the rightmost $4$-bar,
 $x_3^2$ passes the test, remaining a good candidate for being an involutive divisor.
\end{itemize}
So we have

\begin{minipage}{6.5cm}
$\bullet$ $U(x_1^2, x_3)=\{(x_1^2x_3,\emptyset)\}$;\\
$\bullet$  $U(x_1x_2, x_3 )=\{(x_2x_3,\{x_1\})\}$;\\
$\bullet$ $U(x_2^2, x_3 )=\{(x_2x_3,\{x_2\})\}$;\\
$\bullet$  $U(x_1^2x_4, x_3 )=\{( x_3x_4,\{x_1\})\}$;\\  
$\bullet$ $U(x_1x_2x_4, x_3)=\{ (x_3x_4,\{x_1,x_2\})\}$;\\
$\bullet$ $U(x_2^2x_4, x_3 )=\{( x_3x_4,\{x_2\})\}$;\\
$\bullet$  $U(x_1^2x_3, x_3 )=\{(x_3^2,\{x_1 \})\}$;
 \end{minipage}
 \hspace{0.1cm}
 \begin{minipage}{7cm}
 $\bullet$  $U(x_2x_3, x_3   )=\{(x_3^2,\{ x_2\})\}$;\\
$\bullet$ $U(x_3x_4, x_3 )=\{(x_3^2,\{x_4\})\}$;\\
$\bullet$  $U(x_1^2,x_4)=\{(x_1^2x_4,\emptyset),(x_4, \{x_1\})\}$;\\
$\bullet$  $U(x_1x_2,x_4)=\{(x_1x_2x_4,\emptyset)\}$;\\
$\bullet$  $U(x_2^2,x_4)=\{(x_2^2x_4,\emptyset)\}$;\\
 $\bullet$  $U(x_2,x_4)=\{(x_4, \{x_2\})\}$.
\end{minipage}

We continue  choosing $x_2$ as next variable
 and we get:
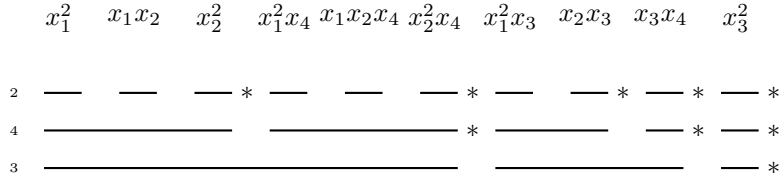
\begin{center}
\begin{tikzpicture}
\node at (4.2,0) [] {${\small x_1^2}$};
\node at (5.2,0) [] {${\small x_1x_2}$};
\node at (6.2,0) [] {${\small x_2^2}$};
\node at (7.2,0) [] {${\small x_1^2x_4}$};
\node at (8.2,0) [] {${\small x_1x_2x_4}$};
\node at (9.2,0) [] {${\small x_2^2x_4}$};
\node at (10.2,0) [] {${\small x_1^2x_3}$};
\node at (11.2,0) [] {${\small x_2x_3}$};
\node at (12.2,0) [] {${\small x_3x_4}$};
\node at (13.2,0) [] {${\small x_3^2}$};

 \draw [thick] (4,-1)--(4.5,-1);
 \draw [thick] (5,-1)--(5.5,-1);
 \draw [thick] (6,-1)--(6.5,-1);
 \draw [thick] (7,-1)--(7.5,-1);
 \draw [thick] (8,-1)--(8.5,-1);
 \draw [thick] (9,-1)--(9.5,-1);
 \draw [thick] (10,-1)--(10.5,-1);
 \draw [thick] (11,-1)--(11.5,-1);
 \draw [thick] (12,-1)--(12.5,-1);
 \draw [thick] (13,-1)--(13.5,-1);

 \draw [thick] (4,-1.5)--(6.5,-1.5);
 \draw [thick] (7,-1.5)--(9.5,-1.5);
\draw [thick] (10,-1.5)--(11.5,-1.5);
\draw [thick] (12,-1.5)--(12.5,-1.5);
\draw [thick] (13,-1.5)--(13.5,-1.5);

\draw [thick] (4,-2) --(9.5,-2);
\draw [thick] (10,-2) --(12.5,-2);
\draw [thick] (13,-2) --(13.5,-2);

\node at (3.6,-1) [] {${\scriptscriptstyle 2}$};

\node at (3.6,-1.5) [] {${\scriptscriptstyle 4}$};
\node at (3.6,-2) [] {${\scriptscriptstyle 3}$};

\node at (9.7,-1) [] {${*}$};
\node at (12.7,-1) [] {${*}$};
\node at (13.7,-1) [] {${*}$};
\node at (11.7,-1) [] {${*}$};
\node at (6.7,-1) [] {${*}$};

\node at (9.7,-1.5) [] {${*}$};
\node at (12.7,-1.5) [] {${*}$};
\node at (13.7,-1.5) [] {${*}$};
\node at (13.7,-2) [] {${*}$};

\end{tikzpicture}
\end{center}
This way, all the $2$-bars are unitary. We check on the 
$2$-bars to have nonincreasing exponents for $x_1$ and this is true. Moreover, 
we check that $x_2$ is multiplicative where it is marked, i.e. for 
$x_2x_3,x_3x_4$ but it clearly holds.
The set $M$ is complete for  $x_1<x_2<x_4<x_3$ and its final Bar Code w.r.t. 
the chosen ordering is
\begin{center}
\begin{tikzpicture}
\node at (4.2,0) [] {${\small x_1^2}$};
\node at (5.2,0) [] {${\small x_1x_2}$};
\node at (6.2,0) [] {${\small x_2^2}$};
\node at (7.2,0) [] {${\small x_1^2x_4}$};
\node at (8.2,0) [] {${\small x_1x_2x_4}$};
\node at (9.2,0) [] {${\small x_2^2x_4}$};
\node at (10.2,0) [] {${\small x_1^2x_3}$};
\node at (11.2,0) [] {${\small x_2x_3}$};
\node at (12.2,0) [] {${\small x_3x_4}$};
\node at (13.2,0) [] {${\small x_3^2}$};

\draw [thick] (4,-0.5)--(4.5,-0.5);
 \draw [thick] (5,-0.5)--(5.5,-0.5);
 \draw [thick] (6,-0.5)--(6.5,-0.5);
 \draw [thick] (7,-0.5)--(7.5,-0.5);
 \draw [thick] (8,-0.5)--(8.5,-0.5);
 \draw [thick] (9,-0.5)--(9.5,-0.5);
 \draw [thick] (10,-0.5)--(10.5,-0.5);
 \draw [thick] (11,-0.5)--(11.5,-0.5);
 \draw [thick] (12,-0.5)--(12.5,-0.5);
 \draw [thick] (13,-0.5)--(13.5,-0.5);

\draw [thick] (4,-1)--(4.5,-1);
 \draw [thick] (5,-1)--(5.5,-1);
 \draw [thick] (6,-1)--(6.5,-1);
 \draw [thick] (7,-1)--(7.5,-1);
 \draw [thick] (8,-1)--(8.5,-1);
 \draw [thick] (9,-1)--(9.5,-1);
 \draw [thick] (10,-1)--(10.5,-1);
 \draw [thick] (11,-1)--(11.5,-1);
 \draw [thick] (12,-1)--(12.5,-1);
 \draw [thick] (13,-1)--(13.5,-1);

 \draw [thick] (4,-1.5)--(6.5,-1.5);
 \draw [thick] (7,-1.5)--(9.5,-1.5);
\draw [thick] (10,-1.5)--(11.5,-1.5);
\draw [thick] (12,-1.5)--(12.5,-1.5);
\draw [thick] (13,-1.5)--(13.5,-1.5);

\draw [thick] (4,-2) --(9.5,-2);
\draw [thick] (10,-2) --(12.5,-2);
\draw [thick] (13,-2) --(13.5,-2);

\node at (3.6,-0.5) [] {${\scriptscriptstyle 1}$};

\node at (3.6,-1) [] {${\scriptscriptstyle 2}$};

\node at (3.6,-1.5) [] {${\scriptscriptstyle 4}$};
\node at (3.6,-2) [] {${\scriptscriptstyle 3}$};

\node at (9.7,-0.5) [] {${*}$};
\node at (12.7,-0.5) [] {${*}$};
\node at (13.7,-0.5) [] {${*}$};
\node at (11.7,-0.5) [] {${*}$};
\node at (6.7,-0.5) [] {${*}$};

\node at (4.7,-0.5) [] {${*}$};
\node at (5.7,-0.5) [] {${*}$};
\node at (7.7,-0.5) [] {${*}$};
\node at (8.7,-0.5) [] {${*}$};
\node at (10.7,-0.5) [] {${*}$};

\node at (9.7,-1) [] {${*}$};
\node at (12.7,-1) [] {${*}$};
\node at (13.7,-1) [] {${*}$};
\node at (11.7,-1) [] {${*}$};
\node at (6.7,-1) [] {${*}$};

\node at (9.7,-1.5) [] {${*}$};
\node at (12.7,-1.5) [] {${*}$};
\node at (13.7,-1.5) [] {${*}$};
\node at (13.7,-2) [] {${*}$};
\end{tikzpicture}
\end{center}

\end{example}

\appendix
\section{Pseudocode of all procedures}\label{Codice}
\begin{algorithm}
\caption{Procedure to generate the candidate list for the current maximal variable (subroutine).}
\begin{algorithmic}[H]\label{candidate1}
\Procedure{CandidateVar}{$M,C$} \Comment{$M$ is a set of terms; $C$ is a set of variables.}
\State {\bf for }$i=1,..., \vert C \vert$ {\bf do}
\State $D_i:=\{\beta \in \NN \vert \exists t \in M,\, \deg_{C[i]}(t)=\beta\}$
\State {\bf end for}
\If{for some $\gamma_1 \in D_i$, $\gamma_1< \max(D_i)$, $\gamma_1+1  
\notin D_i$}
\State Delete $C[i]$ from $C$
\EndIf
\Return $C$\;
\EndProcedure
\end{algorithmic}
\end{algorithm}
\vspace{-0.5cm}
\begin{algorithm}
\caption{Procedure to generate the candidate list for the current maximal
variable.}
\begin{algorithmic}[H]\label{candidate1}
\Procedure{Candidates}{$M,C$} \Comment{$M$ is a list of lists of terms; $C$ is a
set of variables.}
\State {\bf for }$i=1,..., \vert M \vert$ {\bf do}
\State $Y[i]:=$CandidateVar($M[i],C$);
 \State {\bf end for}
\State {\bf return} $\bigcap_{i}Y[i]$\;
\EndProcedure
\end{algorithmic}
\end{algorithm}
\newpage
\begin{algorithm}
\caption{Friends}\label{friends}
\begin{algorithmic}[H]
\Procedure{Friends}{$A,Y,x_j,T$} \Comment{$T$ is the output of a previous execution of Friends (or it is empty), so it is formed by sets of the form $T(t,x_j)$, $t$ terms in the given set, and $x_j$ variables.}
\State {\bf for} $B \in A$ {\bf do}
\State $B'=$NextB$(B,A)$ \Comment{Subroutine taking as input a bar $B$ and the (partial) Bar Code $A$ containing it and giving as output the bar in $A$ that is just on the right of $B$ or error if there is not such bar.}
\State {\bf for} $t \in B$ {\bf do} 
  \State $U(t,x_j)=\{(u,\alpha) \vert u \in B' \textrm{ and } \alpha : tx_j=um,\, m \in \kT[\alpha]\}$
\State {\bf end for}  
\If{$U(t,x_j)=\emptyset$} {\bf return} fail
\EndIf
\State {\bf end for} 
\If {$T \neq \emptyset$} 
\State {\bf for} $B \in A$   {\bf do}
\State {\bf for} $t \in B$  {\bf do}
\State {\bf for} $y \in X \setminus Y$  {\bf do}
\State $U(t,y)=\emptyset$
\State {\bf for} $(U,\alpha) \in T(t,y)$  {\bf do}
\If{$x_j \notin \alpha$}
\State $U(t,y)=\{(U, \alpha)\}\cup U(t,y)$
\Else
\If{$x_j \in \alpha$ and Star$(x_j,t)=$true }
\State  $U(t,y)=\{(U, \alpha)\}\cup U(t,y)$
\EndIf
\EndIf
\State {\bf end for}
\If{$U(t,y)=\emptyset$}
\State {\bf return} fail
\EndIf
\State {\bf end for}
\State {\bf end for}
\State {\bf end for}
\EndIf
\State {\bf return} $U$  
\EndProcedure
\end{algorithmic}
\end{algorithm}
\newpage
\begin{algorithm}
\caption{Common}\label{common}
\begin{algorithmic}[H]
\Procedure{Common}{$A,X,x_i,T$} \Comment{$T$ is the output of a previous execution of Friends (or it is empty), so it is formed by sets of the form $T(t,x_j)$, $t$ terms in the given set, and $x_j$ variables.}
\State $Y=X\setminus \{x_i\}$
\State $Y'=$Candidate$(A,Y)$
\If {$\vert Y' \vert =0$}  {\bf return } $\emptyset$
\EndIf
\State {\bf for} $x_j \in Y'$ {\bf do}
\State {\bf for} $B \in A$ {\bf do}
\State construct $B^j=\{B^{(j}\}$;
 $C=\{B^{(j)} \vert B \in A\}$;
 $U=$Friends$(C,Y,x_j,T)$
\If {$U=$ failure} {\bf continue } 
\EndIf
\If{ $\vert B^{(j)} \vert=1,\,\forall B^{(j)}\in C $ }
 $ord=Y\cup \{x_j\}$
 {\bf return} $ord$
\EndIf
\If{ $Y \neq \emptyset$ }
\State $C=$Common$(C,Y,x_j,U)$
\Else
\If{$Y=\emptyset$} $ord=ord\cup\{x_j\}$
\State {\bf return} ord
\EndIf
\If{ $ord\neq \emptyset$}  $ord=ord \cup \{x_j\}$
\Else $\quad$ {\bf continue}
\EndIf
\EndIf
\State {\bf end for}
\State {\bf end for}
\State {\bf return} $\emptyset$  
\EndProcedure
\end{algorithmic}
\end{algorithm}
\begin{algorithm}
\caption{Ordering}\label{ordering}
\begin{algorithmic}[H]
\Procedure{Ordering}{$M,X$} \Comment{$M$ is a given set of terms; $X$ is the set of all variables}
\State $Y= $Candidate$(M,X)$
\State  {\bf for} $x_i \in Y$ {\bf do}
\State $A=\{A^{(i)}\}$ \Comment{The bars according to $\deg_i$ of the Bar Code we construct.}
\State $T=$Friends$(A,X,x_i,\emptyset)$
\If{ $T=$ failure}
\State {\bf continue}
\EndIf
\If{ $\vert A^{(i)} \vert=1,\,\forall A^{(i)}\in A $ } \Comment{Unitary case.}
\State $ord=\textrm{Append}(X\setminus \{x_i\}, x_i)$ \Comment{We append $x_i$ to the set $X\setminus \{x_i\}$.}
\State {\bf return} $ord$
\EndIf
\State $C=$Common$(A,X,x_i,T)$
\State {\bf return} the variable ordering.
\If{ $C \neq \emptyset$ }
\State $ord=C \cup \{x_i\}$
\State {\bf return} $ord$
\Else $\quad$ {\bf continue}
\EndIf
\State {\bf end for}
\State {\bf return} $\emptyset$
\EndProcedure
\end{algorithmic}
\end{algorithm}
\end{document}